\documentclass[reqno, 11pt]{amsart}
\usepackage[all]{xy} \RequirePackage{color}
 \RequirePackage{amsgen}
\RequirePackage{amstext} \RequirePackage{amsbsy}
\RequirePackage{amsopn} \RequirePackage{amsthm}
\RequirePackage{amssymb} \RequirePackage{epsfig}
\RequirePackage{amscd}

\usepackage[mathscr]{eucal}
\let\eps\varepsilon
\newtheorem{lemma}{Lemma}

\newtheorem{theorem}{Theorem}

\theoremstyle{definition}

\newtheorem{conj}{Conjecture}

\theoremstyle{remark}
\newtheorem{rem}{Remark}

\def\C{\mathbb C}

\def\Z{{\mathbb Z}}

\def\G{\Gamma}

\title{Hypergeometric evaluation identities and  supercongruences}
\author{Ling Long}
\address{Mathematics Department, Iowa State University, Ames, Iowa, 50011, USA}
\email{linglong@iastate.edu} 
\date{Sep. 17, 2010}
\thanks{The research was supported  by an NSA grant
H98230-08-1-0076}
\begin{document}

\begin{abstract}
  In this article, we provide an application of hypergeometric evaluation identities, including a strange valuation of Gosper,
  to prove several supercongruences related to special valuations of truncated hypergeometric
  series. In particular, we prove a conjecture of van Hamme.  

\end{abstract} \maketitle
\section{Introduction} In this article, we use $p$ to denote an odd
prime. In \cite{Zudilin09-Ramanujan}, Zudilin proved several
Ramanujan-type supercongruences using the Wilf-Zeilberger (WZ)
method. One of them, conjectured by van Hamme, is of the form
 \begin{equation}\label{eq:0}
  \sum_{k=0}^{\frac{p-1}{2}} (4k+1) \left (\frac{(\frac12)_k}{k!}\right )^3(-1)^k\equiv (-1)^{\frac{p-1}{2}}p \mod
  p^3,
\end{equation} where $(a)_k=a(a+1)\cdots (a+k-1)$ is the rising factorial when $a\in \mathbb C$ and $k\in \mathbb N$.

The first proof of \eqref{eq:0} was given by Mortenson in
\cite{Mortenson08}. It is said to be of Ramanujan-type because it is
 a $p$-adic version of the following formula of Ramanujan.
$$ \sum_{k=0}^{\infty} (4k+1) \left (\frac{(\frac12)_k}{k!}\right )^3(-1)^k= \frac{2}{\pi}\,.$$
See \cite{Zudilin09-Ramanujan} for more Ramanujan-type
supercongruences.

In this short note, we will  present a new proof of \eqref{eq:0},
which summarizes our strategy in proving similar type of
supercongruences.

In \cite{MO08}, McCarthy and Osburn proved the following conjecture
of van Hamme \cite{vanHamme97}.
$$\sum_{k=0}^{\frac{p-1}2} (4k+1) \left (\frac{(\frac12)_k}{k!}\right
)^5\equiv \left \{ \begin{array}{lll} -\frac{p}{\G_p(\frac34)^4} &
\mod p^3 & \text{ if } p \equiv 1 \mod 4\,; \\ 0& \mod p^3 & \text{
if } p\equiv 3 \mod 4\,, \end{array}\right .$$ where $\G_p(\cdot)$
denotes the $p$-adic Gamma function.

Another comparable conjecture of van Hamme   is  as follows: for any
prime $p>3$
\begin{eqnarray}\label{conj:vanHamme}
\sum_{k=0}^{\frac{p-1}{2}} (6k+1) \left
(\frac{(\frac12)_k}{k!}\right )^3 4^{-k}&{\equiv}~(-1)^{\frac{p-1}{2}}p \mod
  p^4.
\end{eqnarray}van Hamme said ``we have no
real explanation for our observations". In our exploration, it
becomes clear that such particular kind of supercongruences reflects
extra symmetries, which we are able to interpret using
hypergeometric evaluation identities. Of course, they can also be
seen from other perspectives, such as the WZ method.

Meanwhile, it is known that some of the truncated hypergeometric
series are related to the number of rational points on certain
algebraic varieties over finite fields and further to coefficients
of modular forms. For instance, based on the result of Ahlgren and
Ono in \cite{AO00}, Kilbourn \cite{Kilbourn06} proved that
\begin{eqnarray}\label{eq:Kilbourn}
\sum_{k=0}^{\frac{p-1}{2}}  \left (\frac{(\frac12)_k}{k!}\right )^4
&\equiv& a_p \mod
  p^3,
\end{eqnarray}where $a_p$ is the $p$th coefficient of a weight 4 modular form
\begin{equation}\label{eq:mf}
  \eta(2z)^4\eta(4z)^4:=q\prod_{n\ge 1} (1-q^{2n})^4(1-q^{4n})^4, \ q=e^{2\pi i z}.
\end{equation}This is one instance of  supercongruences
conjectured by Rodriguez-Villegas \cite{R-V-Hyper-CY} which relate
special truncated hypergeometric series values and coefficients of
Heck eigenforms.  In \cite{McCarthy}, the author proved another
supercongruence of this type and his approach provides a general combinatorial framework for all these congruences.

In this note, we will establish a few supercongruences via mainly
hypergeometric evaluation identities, and  combinatorics. Since there exist many amazing hypergeometric evaluation
identities in the literature,  we expect that our approach can be
used to prove  other interesting congruences.
\medskip

Here is a summary of our results.

\begin{theorem}\label{thm:1} Let $p>3$ be a prime and $r$ be a positive integer. Then
  \begin{equation}\label{eq:1}
  \sum_{k=0}^{\frac{p^r-1}{2}} (4k+1) \left (\frac{(\frac12)_k}{k!}\right )^4\equiv p^r \mod
  p^{3+r}\,.
\end{equation}
\end{theorem}
\begin{theorem}\label{thm:2} Let $p>3$ be a prime. Then
\begin{equation}\label{eq:2}
  \sum_{k=0}^{\frac{p-1}{2}} (4k+1) \left (\frac{(\frac12)_k}{k!}\right )^6\equiv ~p\cdot a_p \mod
  p^4\,.
\end{equation} 
\end{theorem}
\begin{conj}\label{conj:1}Let $p>3$ be a prime and $r$ be a positive integer. Then
 \begin{equation}\label{eq:2'}
  \sum_{k=0}^{\frac{p^r-1}{2}} (4k+1) \left (\frac{(\frac12)_k}{k!}\right )^6 {\equiv}~ p^r\cdot a_{p^r} \mod
  p^{3+r},
\end{equation}where $a_{p^r}$ is the $p^r$th coefficient of \eqref{eq:mf}.
\end{conj}

\begin{theorem}\label{thm:3}van Hamme's conjecture \eqref{conj:vanHamme} is true.

\end{theorem}
\begin{theorem}\label{thm:4}Let $p>3$ be a prime, then
  \begin{eqnarray}\label{eq:8}
\sum_{k=0}^{\frac{p-1}{2}} (6k+1) \left
(\frac{(\frac12)_k}{k!}\right )^3 \frac{(-1)^k}{8^k}&\equiv&
(-1)^{\frac{p^2-1}{8}+\frac{p-1}{2}}p
\mod
  p^2.
\end{eqnarray} 
\end{theorem}

\medskip

The author would like to thank Heng Huat Chan and Wadim Zudilin for
their encouragements, enlightening discussions and valuable
comments. In particular, Wadim Zudilin pointed out to the author a
few useful ideas and the reference \cite{Gessel-Stanton82}. The author further thanks the anonymous referees for their  detailed comments,  including pointing out a reference \cite{Cai},  on an earlier version of this article.

\section{Preliminaries}
\subsection{Hypergeometric series} For any
positive integer $r$,
$$_{r+1}F_r\left [ \begin{array}{cccccccc} a_1,&a_2,&\cdots,&a_{r+1};&z\\&
b_1,& \cdots,&b_{r} \end{array} \right ]=\sum_{k\ge 0}
\frac{(a_1)_k\cdots (a_{r+1})_k}{k!(b_1)_k\cdots (b_{r})_k}z^k, $$
where $(a)_k$ is the rising factorial and $z\in \C$. A
hypergeometric series terminates if it is well-defined and at least
one of the $a_i$'s is a negative integer. We will make use of this
fact to produce various truncated hypergeometric series.

By the  definition of rising factorial,
\begin{equation}\label{eq:12}
\frac{(\frac12)_k}{k!}=2^{-2k} \binom{2k}k.
\end{equation}

\subsection{Gamma function}Let $\G(x)$ denote the usual Gamma function which is defined for
all $x\in \mathbb C$ except non-positive integers. It satisfies some
well-known properties such as $${\G(x+1)}=x{\G(x)}.$$ Thus,
$$(a)_k=\frac{\G(a+k)}{\G(a)}$$ when $\G(a)\neq 0$ and $\G(a+k)$ are defined.

Another formula we need is the Euler's reflection formula
$$\displaystyle \G(x)\G(1-x)=\frac{\pi}{\sin(\pi
x)}.$$

\subsection{Some combinatorics}\label{ss:com}We gather here some results in combinatorics to be used
later. It is the author's pleasure to acknowledge that the approaches
used in \eqref{eq:1/2} to \eqref{eq:13} are due to Zudilin.

Here is a key idea of Zudilin for rising factorials, see also
\cite[Lemma 1]{CLW09}
\begin{align}\label{eq:1/2}
\left (\frac12+\eps \right )_k
&=(\frac12+\eps)(\frac12+\eps+1)\dotsb(\frac12+\eps+k-1)\\
&=\left (\frac12 \right
)_k\biggl(1+2\eps\sum_{j=1}^k\frac1{2j-1}+4\eps^2\sum_{1\le i<j\le
k}^k\frac1{(2i-1)(2j-1)}+O(\eps^3)\biggr).\notag
\end{align}
Hence, $\left (\frac12+\eps \right )_k\left (\frac12-\eps \right
)_k$, as a power series of $\eps^2$ can be expanded as follows:
\begin{equation}\label{eq:CLW}
\left (\frac12+\eps \right )_k\left (\frac12-\eps \right )_k =\left
(\frac12
\right)_k^2\biggl(1-4\eps^2\sum_{j=1}^k\frac1{(2j-1)^2}+O(\eps^4)\biggr).
\end{equation}
Similarly,
\begin{equation}\label{eq:CLW2}
(1+\eps)_k(1-\eps)_k
=(1)_k^2\biggl(1-\eps^2\sum_{j=1}^k\frac1{j^2}+O(\eps^4)\biggr).
\end{equation}
Letting $\eps=- \frac {p^r} 2$ and $\eps= \frac {p^r} 2$
respectively in \eqref{eq:1/2} and taking $k$ to be an integer
between $1$ and $\frac{p^r-1}2$, we obtain that
$$(-1)^k\binom{\frac{p^r-1}2}{k}\equiv \frac{ (\frac12)_k}{k!} \mod p, \quad \binom{\frac{p^r-1}2+k}{k}\equiv \frac{ (\frac12)_k}{k!} \mod p.$$
Similarly, letting $\eps=\frac {p^r} 2$ in \eqref{eq:CLW} and $k$ be
an integer between $1$ and $\frac{p^r-1}2$ we have
\begin{equation}
(-1)^k\binom{\frac{p^r-1}2}{k}\binom{\frac{p^r-1}2+k}{k}\equiv \left
( \frac{ (\frac12)_k}{k!}  \right )^2 \mod p^2.
\end{equation}

\begin{lemma}\label{lem:comiden0}For any positive integer $n>1$,
  \begin{equation}\label{eq:13}
  (2n+1) \sum_{k=0}^{n}   \frac{1}{2k+1} \binom{n}{k}\binom{n+k}{k}(-1)^k \,{=}\, 1.
\end{equation}
\end{lemma}

\begin{proof}
  We use the following partial fraction decomposition
  \begin{equation*}
    \frac{(t-1)(t-2)\cdots(t-n)}{t(t+1)\cdots (t+n)}=\sum_{k=0}^n(-1)^{n-k}\binom{n}{k}\binom{n+k}{k}\frac{1}{t+k}.
  \end{equation*}Letting $t=\frac12$, it becomes $$(-1)^n\frac{2}{2n+1}=2
  \sum_{k=0}^n(-1)^{n-k}\binom{n}{k}\binom{n+k}{k}\frac{1}{1+2k}\,,$$ which is equivalent to the claim of the Lemma.
\end{proof}
\begin{lemma}\label{lem:comiden1}
  Let $n$ be an odd positive integer. Then
  $$\frac{(\frac{3}{2}-\frac{n}{4})_{\frac{n-1}{2}}(1-\frac{n}{2})_{\frac{n-1}{2}}}{(2-\frac{n}{2})_{\frac{n-1}{2}}(1-\frac{n}{4})_{\frac{n-1}{2}}}=(-1)^{\frac{n-1}2}n.$$
\end{lemma}
\begin{proof}
  Using $(a)_k=\frac{\G(a+k)}{\G(a)}$, we have
\begin{eqnarray*}
\frac{(\frac{3}{2}-\frac{n}{4})_{\frac{n-1}{2}}(1-\frac{n}{2})_{\frac{n-1}{2}}}{(2-\frac{n}{2})_{\frac{n-1}{2}}(1-\frac{n}{4})_{\frac{n-1}{2}}}&=&\frac{\G(\frac
32-\frac n4+\frac{n-1}2)\G(\frac12)\G(2-\frac n2)\G(1-\frac n
4)}{\G(\frac 32-\frac n4)\G(1-\frac n 2)\G(\frac32)\G(1-\frac
n4+\frac{n-1}2)}\\
&=&\frac{(1-\frac n 2)}{\frac12}\frac{\frac n 4\cdot  \G(\frac n
4)\G(1-\frac n 4)}{(\frac 12 -\frac n4)\cdot \G(\frac 12 +\frac
n4)\G(\frac 12 -\frac n4)}\\
&=&n\cdot \frac{\sin (\pi /2-\pi n/4)}{\sin (\pi n/4)}\\&=&n\cdot
\cot (\pi n/4)=(-1)^{\frac{n-1}2}n\,.
\end{eqnarray*}
\end{proof}

\begin{lemma}\label{lem:comiden2}Let $n$ be an odd integer. Then
  $$\frac{(\frac{3}{2}-\frac{n}{4})_{\frac{n-1}{2}}}{(2-\frac{n}{2})_{\frac{n-1}{2}}}2^{\frac{n-1}{2}}=(-1)^{\frac{n^2-1}{8}+\frac{n-1}{2}}n.$$
\end{lemma}
\begin{proof}
\begin{eqnarray*}
  \frac{(\frac{3}{2}-\frac{n}{4})_{\frac{n-1}{2}}}{(2-\frac{n}{2})_{\frac{n-1}{2}}}2^{\frac{n-1}{2}}&=&\frac{(3-\frac n2) (5-\frac n2)\cdots \frac
  n2}{(2-\frac n2)(3-\frac n2)\cdots \frac 12}=\text{sgn} \cdot n,
\end{eqnarray*}where $\text{sgn}=(-1)^{\#}$ and $\#$ is the number
of negative terms appearing in the above fraction. It is easy to see
that $\#=\lfloor \frac{\frac n 2 +1}2 \rfloor+\lfloor \frac n
2\rfloor-2\equiv \frac{n^2-1}{8}+\frac{n-1}{2} \mod 2.$
\end{proof}

\begin{lemma}[Cai, \cite{Cai}]\label{thm:Moreley2} For any prime $p>3$ and positive integer $r$,
  \begin{equation}\label{eq:cai}
(-1)^{\frac{p^r-1}2}\binom{{p^r-1}}{\frac{p^r-1}2}\equiv \left (
\frac{ (\frac12)_{\frac{p^r-1}2}}{\frac{p^r-1}2!}  \right )^2 \mod
p^3.
\end{equation}
\end{lemma}
Using \eqref{eq:12}, the congruence \eqref{eq:cai} is equivalent to 
  $$\binom{p^r-1}{\frac{p^r-1}{2}}  \equiv (-1)^{\frac{p^r-1}{2}} 2^{2(p^r-1)} \mod
  p^3.$$
When $r=1$, it is proved by Morley in \cite{Morley}.

\subsection{A generalized harmonic sum}
Let $\displaystyle H_k^{(2)}:=\sum_{j=1}^k \frac1{j^2}$.
\begin{lemma}\label{lem:H_n=0}Let $p>3$ be a prime. We have
$$H_{\frac{p-1}{2}}^{(2)}\equiv 0 \mod p,$$ and
\begin{equation}\label{lem:sumodd^2=0}
\sum_{j=1}^{\frac{p-1}2}\frac{1}{(2j-1)^2}\equiv 0 \mod p.
  \end{equation}
\end{lemma}
\begin{proof}See \cite{Morley}.
\end{proof}
\medskip

Using arguments in \cite{Morley} or elementary congruence, it is easy to see the following Lemma holds.
\begin{lemma}\label{lem:H(n+k)+H(n-k)}Let $p>3$ be a prime, then for
every integer $k$ between 1 and $p-2$
$$H_{k}^{(2)}+H_{p-1-k}^{(2)}\equiv 0 \mod p.$$
\end{lemma}

\medskip

\begin{lemma}\label{lem:thm-key}Let $p>3$ be a prime and $s$ be a positive
integer. Then
\begin{align*}\label{eq:comiden1}
   \sum_{k=0}^{\frac{p-1}{2}} \left (\frac{(\frac12)_k}{k!}\right )^{2s}\cdot H_{2k}^{(2)}\equiv 0 \mod
   p\,.
\end{align*}
\end{lemma}

\begin{proof}

Using the fact that $\displaystyle
(-1)^k\binom{\frac{p-1}2}{k}\equiv \frac{(\frac12)_k}{k!} \mod p$,
we have
\begin{align*}
   \sum_{k=0}^{\frac{p-1}2} \left (\frac{(\frac12)_k}{k!}\right
   )^{2s}H_{2k}^{(2)}&\equiv \sum_{k=0}^{\frac{p-1}2} \binom{\frac{p-1}2}{k}^{2s}H_{2k}^{(2)} \mod p\\
   &= \frac12 \left (  \sum_{k=0}^{\frac{p-1}2} \binom{\frac{p-1}2}{k}^{2s}H_{2k}^{(2)} + \sum_{k=0}^{\frac{p-1}2}
   \binom{\frac{p-1}2}{\frac{p-1}2-k}^{2s}H_{p-1-2k}^{(2)}\right ) \\&= \frac12 \left  (  \sum_{k=0}^{\frac{p-1}2}
   \binom{\frac{p-1}2}{k}^{2s}(H_{2k}^{(2)}+H_{p-1-2k}^{(2)}) \right ) \\
   &\equiv 0 \mod p.
\end{align*}
\end{proof}
\subsection{An elementary $p$-adic analysis}

Let $F(x_1,\cdots, x_t;z)$ be a $(t+1)$-variable formal power
series. For instance, it could be a  scalar multiple of a
terminating hypergeometric series as  follows:
 $$C\cdot \,_{r+1}F_r\left [
\begin{array}{cccccccc} a_1,&a_2,&\cdots,&a_r&-n;&z\\& b_1,&
\cdots,&b_{r-1},&b_{r}
\end{array} \right ].$$ Assume
that by  specifying values $x_i=a_i,i=1,\cdots, t$ and $z=z_0$,
$$F(a_1,\cdots, a_t;z_0)\in \Z_p.$$ Now we fix $z_0$ and deform the
parameters $a_i$  into polynomials $a_i(x)\in \Z_p[x]$ such that
$a_i(0)=a_i$ for all $1\le i\le t$, and assume that the resulting
function $F(a_1(x),\cdots, a_t(x);z_0)$ is a formal power series in
$x^2$ with coefficients in $\Z_p$, i.e.
$$F(a_1(x),\cdots, a_t(x);z_0)=A_0+A_2x^2+A_4x^4+\cdots, \quad A_i\in \Z_p,$$ where
$A_0=F(a_1,\cdots, a_t;z_0).$
\begin{lemma}\label{lem:congr}
Under the above setting, if  $p^s\mid A_2$ for $s=1,2$, then
$$F(a_1(p),\cdots, a_t(p);z_0)\, \equiv \, A_0\,  \mod p^{2+s}.$$
\end{lemma}


\section{A new proof of \eqref{eq:0}}

Now we briefly outline our method for proving the next few
supercongruences, which is motivated by  the papers \cite{MO08}
   and \cite{Mortenson08}. To
each congruence, we first identify a corresponding hypergeometric
evaluation identity, which with specified parameters is congruent to
a target truncated hypergeometric series evaluation up to some power
of $p$. Usually the power of $p$ hence obtained is weaker than the
conjectural exponent. In our cases, we reduce the optimal
congruences to some congruence combinatorial identities, which are
established using additional hypergeometric evaluation identities or
combinatorics.

Our strategy can be best implemented in the following new proof of
\eqref{eq:0}. An identity of Whipple \cite[(5.1)]{Whipple26} says
 $$_4F_3\left [ \begin{array}{cccccccc} a,&1+a/2,&c,&d;&-1\\&
a/2,& 1+a-c,&1+a-d \end{array} \right
]=\frac{\G(1+a-c)\G(1+a-d)}{\G(1+a)\G(1+a-c-d)}. $$ Letting
$a=\frac12, c=\frac12+\frac{p}{2},d=\frac12-\frac{p}{2}$, we
conclude immediately that
\begin{equation*}
  \sum_{k=0}^{\frac{p-1}{2}} (4k+1) \left (\frac{(\frac12)_k}{k!}\right )^3(-1)^k \equiv \frac{\G(1-\frac{p}{2})
  \G(1+\frac{p}{2})}{\G(\frac12)\G(\frac{3}{2})}=(-1)^{\frac{p-1}2}p \mod
  p^2.
\end{equation*} To achieve the congruence modulo $p^3$, we consider the expansion of the terminating hypergeometric series (it terminates as $\frac{1-p}2$ is
a negative integer): \begin{equation}\label{eq:10}_4F_3 \left [
\begin{array}{cccccccc} \frac{1-p}{2},& \frac{5}4,&\frac{1-x}{2},
&\frac{1+x}{2}; &-1\\&\frac{1}{4},&1+\frac{x}2, &1-\frac{x}2
\end{array} \right ]=\sum_{k=0}^{\frac{p-1}{2}} (4k+1) \left
(\frac{(\frac12)_k}{k!}\right )^3(-1)^k\,+\, A_2x^2+\cdots ,
\end{equation} for some $A_2\in \Z_p$.

By Lemma \ref{lem:congr}, if  $p\mid A_2$, we are done. Now we
follow Mortenson \cite{Mortenson08} to use another hypergeometric
evaluation identity, which is a specialization of Whipple's $_7F_6$
formula (see \cite[pp. 28]{Bailey}).
\begin{eqnarray*}
 && _6F_5 \left [ \begin{array}{cccccccc} a,& 1+\frac{a}2,&b, &c,&d,&e; &-1\\&\frac{a}{2},&1+a-b, &1+a-c,& 1+a-d,&1+a-e \end{array} \right ]
 \\&& =\frac{\G(1+a-d)\G(1+a-e)}{\G(1+a)\G(1+a-d-e)}\, _3F_2 \left [ \begin{array}{cccccccc} 1+a-b-c,&d, &e; &1\\ &1+a-b, &1+a-c\end{array} \right ].
\end{eqnarray*}Letting $a=\frac12, b=\frac{1-x}{2},c=\frac{1+x}{2},e=\frac{1-p}2,d=1$, we have
\begin{eqnarray}\label{eq:6F5}
 _6F_5 \left [ \begin{array}{cccccccc} \frac12,& \frac{5}4,&\frac{1-x}{2}, &\frac{1+x}{2},&\frac{1-p}{2},&1; &-1
 \\&\frac{1}{4},&1+\frac{x}2, &1-\frac{x}2,& \frac12,&1+\frac{p}2 \end{array} \right ]
  =\frac{\G(\frac12)\G(1+\frac{p}{2})}{\G(\frac32)\G(\frac{p}2)}\, _3F_2 \left [ \begin{array}{cccccccc} \frac12,&1,  &\frac12-\frac{p}{2};
  &1\\ &1+\frac{x}2, &1-\frac{x}2\end{array} \right ].
\end{eqnarray} Since
$\displaystyle
\frac{\G(\frac12)\G(1+\frac{p}{2})}{\G(\frac32)\G(\frac{p}2)}=p,$
every $x$-coefficient  of the above is in $p\Z_p$. Moreover,
modulo $p$ the left hand side of \eqref{eq:10} is congruent to that of \eqref{eq:6F5}. So when we expand the left hand side of
\eqref{eq:10} in terms of $x$, the  coefficients are all in $p\Z_p$.
In particular, $p\mid A_2$ and this concludes the
proof of \eqref{eq:0}.

\section{Proofs of Theorems \ref{thm:1}, \ref{thm:2}, \ref{thm:3}, and \ref{thm:4}}
We first recall the following identity of Whipple
\cite[(7.7)]{Whipple26}:
\begin{eqnarray}\label{eq:W7.7}
\begin{split}
  _7F_6 \left [ \begin{array}{cccccccc} a,& 1+\frac12 a,&c,&d,&e,&f,&g;& 1 \\ & \frac{1}{2}a, & 1+a-c,& 1+a-d, & 1+a-e,& 1+a-f& 1+a-g; &
  \end{array} \right ]\\=\frac{\G(1+a-e)\G(1+a-f)\G(1+a-g)\G(1+a-e-f-g)}{\G(1+a)\G(1+a-f-g)\G(1+a-e-f)\G(1+a-e-f)}\times
  \\  _4F_3 \left [ \begin{array}{cccccccc} 1+a-c-d,& e,&f,&g; &1\\&e+f+g-a, &1+a-c,& 1+a-d \end{array} \right ],
\end{split}
  \end{eqnarray}subject to the condition that the $_4F_3$ is a terminating series.

\subsection{Proof of Theorem \ref{thm:1}}

Let  $r$ be a positive integer and $p>3$ a prime. In the identity
\eqref{eq:W7.7}, we let $$a=\frac12, c=\frac12+i\frac{p^r}{2},
d=\frac12-i\frac{p^r}{2}, e=\frac12+\frac{p^r}{2},
f=\frac12-\frac{p^r}{2}, g=1,$$ where $i=\sqrt{-1}$ and thereafter, then following McCarthy and Osburn's
argument we know the left hand side of \eqref{eq:W7.7} is congruence to
$$\sum_{k=0}^{\frac{p^r-1}{2}}  (4k+1) \left (\frac{(\frac12)_k}{k!}\right
)^4 \mod p^{4r}$$ and the right hand side of \eqref{eq:W7.7} equals
$$\frac{\G(1-\frac{p^r}{2})\G(1+\frac{p^r}{2})\G(-\frac12)}{\G(\frac{3}{2})\G(-\frac{p^r}{2})\G(\frac{p^r}{2})}\, _4F_3\left [ \begin{array}{cccccccc} \frac12,&\frac12+\frac{p^r}{2},&\frac12-\frac{p^r}{2},&1;&1\\&
\frac{3}{2},& 1-i\frac{p^r}{2},&1+i\frac{p^r}{2} \end{array} \right
]. $$ Since
$$\frac{\G(1-\frac{p^r}{2})\G(1+\frac{p^r}{2})\G(-\frac12)}{\G(\frac{3}{2})\G(-\frac{p^r}{2})\G(\frac{p^r}{2})}=p^{2r},$$
it suffices to prove
\begin{equation*}\label{eq:5'}
  p^r\cdot  \sum_{k= 0}^{\frac{p^r-1}{2}}   \frac{1}{2k+1} \left (\frac{(\frac12)_k}{k!}\right )^2 {\equiv} \ 1 \mod
  p^3\,,  \text{ for }p>3.
\end{equation*}

 Recall that Lemma \ref{lem:comiden0} says
for any odd integer $n>1$,
  \begin{equation*}
  (2n+1) \sum_{k=0}^{n}   \frac{(-1)^k}{2k+1} \binom{n}{k}\binom{n+k}{k} \,{=}\, 1.
\end{equation*}Therefore, combining this identity, congruence \eqref{eq:CLW}, and Lemma \ref{thm:Moreley2}, we have
\begin{eqnarray*}\label{eq:5'}
  p^r\cdot  \sum_{k= 0}^{\frac{p^r-1}{2}}   \frac{1}{2k+1} \left (\frac{(\frac12)_k}{k!}\right )^2
  &=&p^r\cdot  \sum_{k= 0}^{\frac{p^r-1}{2}-1}   \frac{1}{2k+1} \left (\frac{(\frac12)_k}{k!}\right
  )^2+\left (\frac{(\frac12)_{\frac{p^r-1}{2}}}{{\frac{p^r-1}{2}}!}\right
  )^2\\
  &\equiv& p^r\cdot \sum_{k=0}^{\frac{p^r-1}{2}-1}   \frac{(-1)^k}{2k+1} \binom{\frac{p^r-1}{2}}{k}\binom{\frac{p^r-1}{2}+k}{k} +(-1)^{\frac{p^r-1}2}\binom{p^r-1}{\frac{p^r-1}2} \mod p^3\\
  &\equiv &1 \mod p^3\,.
\end{eqnarray*}
This concludes our proof of Theorem \ref{thm:1}.

\subsection{Proof of Theorem \ref{thm:2}} In the formula
\eqref{eq:W7.7}, take $$a=\frac12,c=\frac12+i\frac{p}{2},
d=\frac12-i\frac{p}{2}, e=\frac12-\frac{p}{2},
f=\frac12+\frac{p}{2}, g=\frac12-p^4,$$  then the left hand side of
\eqref{eq:W7.7} is congruent to
$$\sum_{k=0}^{\frac{p-1}2}  (4k+1) \left (\frac{(\frac12)_k}{k!}\right
)^6 \mod p^4.$$ Meanwhile, the right hand side of \eqref{eq:W7.7} is
congruent to
\begin{eqnarray*}
\frac{\G(1-\frac{p}{2})\G(1+\frac{p}{2})}{\G(\frac12)\G(\frac{3}{2})}\frac{\G(1+p^4)\G(p^4)}{\G(\frac12+\frac{p}{2}+p^4)\G(\frac12-\frac{p}{2}+p^4)}
 \sum_{k=0}^{\frac{p-1}{2}}
   \frac{(\frac12)_k^2(\frac12+\frac{p}{2})_k(\frac12-\frac{p}{2})_k}{k!^2(1-i\frac{p}{2})_k(1+i\frac{p}{2})_k} \mod p^4,
\end{eqnarray*} where
 $$\frac{\G(1-\frac{p}{2})\G(1+\frac{p}{2})}{\G(\frac12)\G(\frac{3}{2})}=(-1)^{\frac{p-1}{2}} p\,,$$ and
\begin{eqnarray*}\frac{\G(1+p^4)\G(p^4)}{\G(\frac12+\frac{p}{2}+p^4)\G(\frac12-\frac{p}{2}+p^4)}&=&
\frac{(p^4-\frac{p-1}2)_{\frac{p-1}2}}{(1+p^4)_{\frac{p-1}2}}\\ &\equiv&\frac{(-\frac{p-1}2)(-\frac{p-1}2+1)\cdots (-1)}{1\cdot 2\cdots (\frac{p-1}2)} \mod p=(-1)^{\frac{p-1}{2}}.\end{eqnarray*}

Therefore, Theorem \ref{thm:2} follows from the  result of
 Kilbourn (see \eqref{eq:Kilbourn}) and the next Lemma.
\begin{lemma}\label{lem: thm1 mod p^3}Let $p>3$ be a prime, then
  \begin{eqnarray*}\label{eq:conj2}
   \sum_{k=0}^{\frac{p-1}{2}}
   \frac{(\frac12)_k^2(\frac12+\frac{p}{2})_k(\frac12-\frac{p}{2})_k}{k!^2(1-i\frac{p}{2})_k(1+i\frac{p}{2})_k}\equiv \sum_{k=0}^{\frac{p-1}{2}}
   \left (\frac{(\frac12)_k}{k!}\right )^4 \mod p^3.
 \end{eqnarray*}
\end{lemma}
\begin{proof}
Expand
\begin{eqnarray*}
 \sum_{k=0}^{\frac{p-1}{2}}
   \frac{(\frac12)_k^2(\frac12+\frac{x}{2})_k(\frac12-\frac{x}{2})_k}{k!^2(1-xi/2)_k(1+xi/2)_k}= \sum_{k=0}^{\frac{p-1}{2}}
   \left (\frac{(\frac12)_k}{k!}\right )^4(1 +b_{2,k}
   x^2+b_{4,k}x^4+\cdots).
   \end{eqnarray*} Using  \eqref{eq:CLW} and \eqref{eq:CLW2}, we have $$b_{2,k}=-
\sum_{j=1}^k\frac1{(2j-1)^2}- \frac14 \sum_{j=1}^k\frac1{j^2}=-
\sum_{j=1}^{2k}\frac1{j^2}.$$

The claim of the Lemma is valid by using  Lemma \ref{lem:congr} and taking $s=2$ in
Lemma \ref{lem:thm-key}.  
\end{proof}

\subsection{The proof of Theorem \ref{thm:3}} We start with the
following combinatorial identity.
\begin{lemma}\label{lem:10}
  \begin{eqnarray*}
  \sum_{k=0}^{\frac{p-1}{2}} (6k+1)\frac{(\frac12)_k(\frac12-\frac{p}{2})_k(\frac12+\frac{p}{2})_k}
  {(1)_k(1+\frac{p}{4})_k(1-\frac{p}{4})_k}\frac{1}{4^k}=(-1)^{\frac{p-1}{2}}p \,.
\end{eqnarray*}
\end{lemma}

\begin{proof}

Recall that (31.1) of Gessel \cite{Gessel95} says
\begin{eqnarray*}
 _5F_4 \left [ \begin{array}{cccccccc} \frac12+a-c,& -n,&n+1,&2-2c+n,&\frac{5}{3}-\frac{2c}{3}+\frac{n}{3}; &\frac14\\&2-c+n,&\frac{2}{3}-\frac{2c}{3}+\frac{n}{3}, &n-2a+2,& \frac{3}{2}-c \end{array} \right ]
  =\frac{(2-c)_n(2-2a)_n}{(3-2c)_n(\frac{3}{2}-a)_n}.
\end{eqnarray*}

Letting
$a=\frac12+\frac{p}{4},c=\frac12+\frac{p}{4},n=\frac{p-1}{2}$ and
using  Lemma \ref{lem:comiden1}, we have
\begin{eqnarray}\label{eq:thm3-1}
 _5F_4 \left [ \begin{array}{cccccccc} \frac12,& \frac12,&\frac76&,\frac12-\frac{p}{2},&\frac12+\frac{p}{2}; &\frac14\\&\frac12,
 &\frac16 &1-\frac{p}{4},& 1+\frac{p}{4} \end{array} \right
  ]=\frac{(\frac{3}{2}-\frac{p}{4})_{\frac{p-1}{2}}(1-\frac{p}{2})_{\frac{p-1}{2}}}{(2-\frac{p}{2})_{\frac{p-1}{2}}(1-\frac{p}{4})_{\frac{p-1}{2}}}=(-1)^{\frac{p-1}2}p.
\end{eqnarray}
\end{proof}

\begin{lemma}
  The function \begin{eqnarray*}
  \left (\sum_{k=0}^{\frac{p-1}{2}} (6k+1)\frac{(\frac12)_k(\frac12-\frac{x}{2})_k(\frac12+\frac{x}{2})_k}
  {(1)_k(1+\frac{x}{4})_k(1-\frac{x}{4})_k}\frac{1}{4^k}\right )\big/ \left (\sum_{k=0}^{\frac{p-1}{2}}
  \frac{6k+1}{4^k} \left (\frac{(\frac12)_k}{k!} \right
  )^3 \right )
\end{eqnarray*} is a formal power series in $x^2$ with coefficients in $\Z_p$. Its $x^2$ coefficient is zero modulo $p$.

\end{lemma}

\begin{proof}
We use the following \emph{strange} valuation of Gosper  (cf.
\cite[(1.2]{Gessel-Stanton82})
\begin{eqnarray*}
  _5F_4 \left [ \begin{array}{cccccccc} 2a,& 2b,&1-2b,&1+\frac{2a}{3},&-n; &\frac14\\&a+b-1,&a+b+\frac12, &\frac{2a}{3},& 1+2a+2n \end{array} \right ]
  =\frac{(a+\frac12)_n(a+1)_n}{(a+b+\frac12)_n(a-b+1)_n}.
\end{eqnarray*}Letting
$a=\frac14,b=\frac14-\frac{x}{4},n=\frac{p-1}{2}$,
 the left hand side of the above equals
\begin{eqnarray}\label{cor:thm3}
  _5F_4 \left [ \begin{array}{cccccccc} \frac12,& \frac12-\frac{x}2,& \frac12+\frac{x}2,&\frac76,&\frac12-\frac{p}{2}; &\frac14
 \\&\frac12+p,&\frac16 &1-\frac{x}{4},& 1+\frac{x}{4}
 \end{array} \right
  ] =\frac{(\frac{3}{4})_{\frac{p-1}{2}}(\frac54)_{\frac{p-1}{2}}}{(1-\frac x4)_{\frac{p-1}{2}}(1+\frac x4)_{\frac{p-1}{2}}}.
\end{eqnarray}
We remark that \begin{equation}\label{eq:thm3-cong}_5F_4 \left [
\begin{array}{cccccccc} \frac12,& \frac12-\frac{x}2,&
\frac12+\frac{x}2,&\frac76,&\frac12-\frac{p}{2}; &\frac14
 \\&\frac12+p,&\frac16 &1-\frac{x}{4},& 1+\frac{x}{4}
 \end{array} \right
  ]\equiv \displaystyle \sum_{k=0}^{\frac{p-1}{2}} \frac{6k+1}{4^k}\frac{(\frac12)_k(\frac12-\frac{x}{2})_k(\frac12+\frac{x}{2})_k}
  {(1)_k(1+\frac{x}{4})_k(1-\frac{x}{4})_k} \mod p.
\end{equation}
When $x=0$, the right hand side of \eqref{cor:thm3} equals
$\frac{(\frac{3}{4})_{\frac{p-1}{2}}(\frac54)_{\frac{p-1}{2}}}{(1)_{\frac{p-1}{2}}^2}
$, which is in $p\Z_p$. As a matter of fact, if $p\equiv 1 \mod 4$
then $\frac54+\frac{p-1}{4}-1=\frac{p}4$; and if $p\equiv 3 \mod 4$
then $\frac34+\frac{p-3}4=\frac p4$, while $(1)_{\frac{p-1}{2}}$ is
a $p$-adic unit.  It is not difficult to see that $p$ divides
$\frac{(\frac{3}{4})_{\frac{p-1}{2}}(\frac54)_{\frac{p-1}{2}}}{(1)_{\frac{p-1}{2}}^2}
$ exactly. Consequently, if we expand $ _5F_4 \left [
\begin{array}{cccccccc} \frac12,& \frac12-\frac{x}2,&
\frac12+\frac{x}2,&\frac76,&\frac12-\frac{p}{2}; &\frac14
 \\&\frac12+p,&\frac16 &1-\frac{x}{4},& 1+\frac{x}{4}
 \end{array} \right
  ]$ in terms of formal power series of $x$ (in fact, $x^2$), each coefficient is in $p\Z_p$. Thus
the coefficients of the right hand side of \eqref{eq:thm3-cong},
including the coefficient of $x^2$, are all divisible by $p$. By
Lemmas \ref{lem:congr} and \ref{lem:10},
\begin{equation}\label{eq:16a}
    \sum_{k=0}^{\frac{p-1}{2}}
  \frac{6k+1}{4^k} \left (\frac{(\frac12)_k}{k!}\right )^3 \equiv (-1)^{\frac{p-1}2}p \mod p^3.
  \end{equation}Namely, $$\sum_{k=0}^{\frac{p-1}{2}}
  \frac{6k+1}{4^k} \left (\frac{(\frac12)_k}{k!} \right
  )^3=(-1)^{\frac{p-1}2}p+ap^3,$$ for some $a\in \Z_p$. The statement of Theorem
  \ref{thm:3} is equivalent to $a \in p\Z_p$.

The quotient \begin{equation}\label{eq:16}
  \left ( \sum_{k=0}^{\frac{p-1}{2}} \frac{6k+1}{4^k}\frac{(\frac12)_k(\frac12-\frac{x}{2})_k(\frac12+\frac{x}{2})_k}
  {(1)_k(1+\frac{x}{4})_k(1-\frac{x}{4})_k}\right ) \big/ \left (\sum_{k=0}^{\frac{p-1}{2}} \frac{6k+1}{4^k}\frac{(\frac12)_k}
  {(1)_k}\right )
  \end{equation} is a formal power series in $x^2$ with $p$-integral coefficients, as  the denominators are divisible by $p$ exactly. The same conclusion
  applies to  the following  \begin{eqnarray*}
   \left ( _5F_4 \left [ \begin{array}{cccccccc} \frac12,& \frac12-\frac{x}2,& \frac12+\frac{x}2,&\frac76,&\frac12-\frac{p}{2}; &\frac14
 \\&\frac12+p,&\frac16 &1-\frac{x}{4},& 1+\frac{x}{4}
 \end{array} \right
  ] \right )\big/  \left ( _5F_4 \left [ \begin{array}{cccccccc} \frac12,& \frac12,& \frac12,&\frac76,&\frac12-\frac{p}{2}; &\frac14
 \\&\frac12+p,&\frac16 &1-\frac{x}{4},& 1+\frac{x}{4}
 \end{array} \right
  ] \right )\\
  = \left ( _5F_4 \left [ \begin{array}{cccccccc} \frac12,& \frac12-\frac{x}2,& \frac12+\frac{x}2,&\frac76,&\frac12-\frac{p}{2}; &\frac14
 \\&\frac12+p,&\frac16 &1-\frac{x}{4},& 1+\frac{x}{4}
 \end{array} \right
  ] \right )\big/  \left (\frac{(\frac{3}{4})_{\frac{p-1}{2}}(\frac54)_{\frac{p-1}{2}}}{(1)_{\frac{p-1}{2}}^2} \right
  )
  =\frac{(1)_{\frac{p-1}{2}}^2}{(1-\frac
x4)_{\frac{p-1}{2}}(1+\frac x4)_{\frac{p-1}{2}}}\,.
  \end{eqnarray*}

On the other hand, by \eqref{eq:CLW2}, the $x^2$
coefficient of $\frac{(1)_{\frac{p-1}{2}}^2}{(1-\frac
x4)_{\frac{p-1}{2}}(1+\frac x4)_{\frac{p-1}{2}}}$ is a scalar
multiple of $H_{\frac{p-1}2}^{(2)}$, which is in $p\Z_p$ by Lemma
\ref{lem:H_n=0}; so is the $x^2$ coefficient of \eqref{eq:16}.
\end{proof}

By Lemma \ref{lem:congr} and the above analysis, $$
\frac{(-1)^{\frac{p-1}2}p}{(-1)^{\frac{p-1}2}p+ap^3}=\frac{(-1)^{\frac{p-1}2}}{(-1)^{\frac{p-1}2}+ap^2}\equiv
1 \mod p^3,$$ hence $a\in p\Z_p$, which concludes the proof of
Theorem \ref{thm:3}.

\subsection{The proof of Theorem \ref{thm:4}}Recall we want to prove
 \begin{equation*}\label{eq:x}
  \sum_{k=0}^{\frac{p-1}{2}} (6k+1) \left (\frac{(\frac12)_k}{k!}\right )^3 \frac{(-1)^k}{8^k}\equiv (-1)^{\frac{p^2-1}{8}+\frac{p-1}{2}}p \mod
  p^2.
\end{equation*}
It is a consequence of the following combinatorial identity.
\begin{lemma}\label{lem:12}
  \begin{eqnarray*}
  \sum_{k=0}^{\frac{p-1}{2}} (6k+1)\frac{(\frac12)_k(\frac12-\frac{p}{2})_k(\frac12+\frac{p}{2})_k}
  {(1)_k(1+\frac{p}{4})_k(1-\frac{p}{4})_k}\frac{(-1)^k}{8^k}=(-1)^{\frac{p^2-1}{8}+\frac{p-1}{2}}p .
\end{eqnarray*}
\end{lemma}

\begin{proof}
This time, we use the following identity in Gessel \cite[pp. 544,
last identity]{Gessel95}
\begin{equation}\label{eq:Gessel95}
  _4F_3 \left [ \begin{array}{cccccccc} 2a+n+1,& n+1,&\frac{2a}{3}+\frac{n}{3}+\frac{4}{3},&-n; &-\frac{1}{8}\\&a+\frac{3}{2}+n, &\frac{2a}{3}+\frac{n}{3}+\frac{1}{3},& 1+a \end{array} \right
  ]=\frac{(a+\frac{3}{2})_n}{(2a+2)_n}2^n.
\end{equation}
Letting $a=-\frac{p}{4}$ and $n=\frac{p-1}{2}$ and using Lemma
\ref{lem:comiden2}, we have
\begin{equation*}
  _4F_3 \left [ \begin{array}{cccccccc} \frac12,& \frac76,&\frac12+\frac{p}{2},&\frac12-\frac{p}{2}; &-\frac{1}{8}\\&\frac16, &1-\frac{p}{4},& 1+\frac{p}{4} \end{array} \right
  ]=\frac{(\frac{3}{2}-\frac{p}{4})_{\frac{p-1}{2}}}{(2-\frac{p}{2})_{\frac{p-1}{2}}}2^{\frac{p-1}{2}}=(-1)^{\frac{p^2-1}{8}+\frac{p-1}{2}}p.
\end{equation*}

\end{proof}

\begin{rem}The following conjecture of van Hamme
  \begin{equation}\label{eq:x}
  \sum_{k=0}^{\frac{p-1}{2}} (6k+1) \left (\frac{(\frac12)_k}{k!}\right )^3 \frac{(-1)^k}{8^k}{\equiv} ~ (-1)^{\frac{p^2-1}{8}+\frac{p-1}{2}}p \mod
  p^3
\end{equation} holds subject to the following:
\begin{align}\label{eq:comconj2}
  \sum_{k=0}^{\frac{p-1}{2}}(6k+1) \left (\frac{(\frac12)_k}{k!}\right )^3
  \left (\sum_{j=1}^{k}\frac1{(2j-1)^2}-\frac{1}{16}\sum_{j=1}^{k}\frac1{j^2}\right )\frac{(-1)^k}{8^k}{\equiv} \, 0 \mod
  p.
\end{align}The proof of
\eqref{eq:comconj2} is left to the interested reader.
\end{rem}

\begin{rem}

Note that using the method in \cite{Zudilin09-Ramanujan},
Zudilin proved the congruence \eqref{conj:vanHamme} modulo $p^2$ and
the congruence \eqref{eq:8} modulo $p$. \end{rem}

\providecommand{\bysame}{\leavevmode\hbox
to3em{\hrulefill}\thinspace}
\providecommand{\MR}{\relax\ifhmode\unskip\space\fi MR }
\providecommand{\MRhref}[2]{%
  \href{http://www.ams.org/mathscinet-getitem?mr=#1}{#2}
} \providecommand{\href}[2]{#2}

\end{document}